\documentclass[11pt]{article}
\usepackage[a4paper,margin=1in]{geometry}
\usepackage{amsmath,amssymb,amsthm}
\usepackage{mathrsfs}
\usepackage{graphicx}
\usepackage{hyperref}
\usepackage{enumitem}
\usepackage{tikz}
\usepackage{float}
\hypersetup{colorlinks=true,linkcolor=blue,citecolor=blue,urlcolor=blue}
\newtheorem{theorem}{Theorem}
\newtheorem{proposition}{Proposition}
\newtheorem{remark}{Remark}
\title{Analytic Continuation of the Frenet Curvature 
Form of an Analytic Plane Curve}
\author{A. G. Nuramatov}
\date{}
\begin{document}
\maketitle
\begin{abstract}
Fermi coordinates provide a natural local coordinate system near a plane curve. Nevertheless, they 
are not conformal and are not generated directly by the curvature form. The present work asks whether 
a real analytic plane curve determines a distinguished conformal coordinate net generated by its curvature.
Let
\[
\omega_F=-k\,ds
\]
be the Frenet curvature 1-form of an analytic curve. We show that its holomorphic continuation
\[
\Omega=K(q)\,dq
\]
naturally determines a conformal coordinate net through the reconstruction formula
\[
z_q=\exp\left(i\int\Omega\right).
\]
The holomorphic differential \(\Omega\) is identified with a 
conformal connection differential associated with the coordinate net and
 whose restriction  to the initial curve coincides with the original curvature form.
For a pair of dual conformal nets, the derivative of the conformal transition map admits the geometric representation
\[
z_q=-\overline{\left(\frac{\widetilde\omega}{\omega}\right)},
\]
where \(\omega\) and \(\widetilde\omega\) 
denote the complex curvature coefficients of the two nets.
 Although neither curvature coefficient is holomorphic in general, 
 their normalized conjugate dual ratio is holomorphic. Thus an analytic plane curve determines, 
 through the holomorphic continuation of its curvature form, a distinguished conformal coordinate net 
 in its neighborhood.
\end{abstract}
\section{Introduction}
Classical mathematical physics traditionally employs coordinate systems adapted to geometry. Classical examples include elliptic, parabolic and bipolar coordinate systems, as discussed by Morse and Feshbach \cite{Morse}. In the classical theory, coordinate nets are usually introduced first, and geometric quantities are subsequently computed from them. The present work approaches the problem from the opposite direction.
Fermi coordinates provide a natural local coordinate system near a plane curve. Nevertheless, they are not conformal and are not generated directly by the curvature.
We ask whether a real analytic plane curve determines a distinguished conformal coordinate net generated by its curvature.
Let
\[
\gamma=\gamma(s)
\]
be a real analytic plane curve parametrized by arc length. Its Frenet curvature 1-form is
\[
\omega_F=-k\,ds.
\]
The central idea of the paper is to regard the Frenet curvature data  as the primary geometric object.
The corresponding holomorphic  continuation is represented by the   holomorphic differential
\[
\Omega=K(q)\,dq.
\]
We show that this continuation determines a conformal coordinate net through the reconstruction formula
\[
z_q=\exp\left(i\int\Omega\right).
\]
This formula may be viewed as a two-dimensional analogue of the classical reconstruction formula
\[
T(s)=\exp\left(i\int k\,ds\right),
\]
which reconstructs the tangent vector of a plane curve from its connection 1-form. The resulting conformal net is distinguished by the analytic continuation of the Frenet curvature 1-form.
 The elliptic coordinate net provides an explicit realization of the construction.
\section{Connection on a Curve}
Let
\[
\gamma=\gamma(s)
\]
be a regular plane curve parametrized by arc length. Let \(T,N\) denote the unit tangent and unit normal
 vectors. The Frenet equations are
\[
dT=kN\,ds,\qquad dN=-kT\,ds.
\]
Equivalently,
\[
d
\begin{pmatrix}
T\\
N
\end{pmatrix}
=
\begin{pmatrix}
0&kds\\
-kds&0
\end{pmatrix}
\begin{pmatrix}
T\\
N
\end{pmatrix}.
\]
Let \(F=(e_1,e_2)\) be an orthonormal frame. Cartan's structure equations give
\[
de_1=\omega^2_1 e_2,\qquad de_2=\omega^1_2 e_1,
\]
where
\[
\omega^2_1=-\omega^1_2.
\]
Comparing with Frenet's equations, we obtain
\[
\boxed{k\,ds=-\omega^1_2.}
\]
Thus the Frenet curvature 1-form is identified with the Cartan connection 1-form of the moving orthonormal frame. Accordingly, throughout the present work the curvature of the curve is viewed as a connection 1-form rather than merely as a scalar function.
The tangent vector is reconstructed by
\[
\boxed{T(s)=\exp\left(i\int k\,ds\right).}
\]
This one-dimensional reconstruction formula will serve as a model 
for the conformal reconstruction developed below.
\section{Conformal Coordinate Nets}

Let
\[
q=q^1+iq^2
\]
be a conformal coordinate. The metric has the form
\[
ds^2
=
H^2\bigl((dq^1)^2+(dq^2)^2\bigr),
\]
where $H(q^1,q^2)$ is the conformal factor.

The orthonormal coframe is
\[
\theta^1=H\,dq^1,
\qquad
\theta^2=H\,dq^2.
\]

Cartan's structure equations imply
\[
\omega^1_{\;2}
=
\omega^1_{\;21}\theta^1
+
\omega^1_{\;22}\theta^2,
\]
where
\[
\omega^1_{\;21}
=
\frac{H_{q^2}}{H^2},
\qquad
\omega^1_{\;22}
=
-\frac{H_{q^1}}{H^2}.
\]

We define the complex curvature coefficient by
\[
\omega
=
-\bigl(\omega^1_{\;21}-i\omega^1_{\;22}\bigr).
\]

\textbf{Lemma 1.}
The complex curvature coefficient admits the logarithmic representation
\[
\omega
=
-\frac{2i}{H}\partial_q\log H,
\]
where
\[
\partial_q
=
\frac12(\partial_{q^1}-i\partial_{q^2}).
\]

\begin{proof}
Substituting the expressions for
$\omega^1_{\;21}$ and $\omega^1_{\;22}$ gives
\[
\omega
=
-
\frac{H_{q^2}+iH_{q^1}}{H^2}
=
-
\frac{i}{H^2}
(H_{q^1}-iH_{q^2}).
\]
Since
\[
\partial_qH
=
\frac12(H_{q^1}-iH_{q^2}),
\]
we obtain
\[
\omega
=
-
\frac{2i}{H^2}\partial_qH
=
-\frac{2i}{H}\partial_q\log H.
\]
\end{proof}

Thus the complex curvature coefficient is naturally expressed as the logarithmic derivative of the 
conformal factor. This motivates the introduction of the conformal connection differential
\[
\Omega
=
H\omega\,dq
=
-2i\,\partial_q\log H\,dq.
\]

In general, $\omega$ is not holomorphic. The differential
\[
\Omega=H\omega\,dq
\]
will play the role of the holomorphic continuation of the Frenet curvature form. The following sections show that
\[
\Omega
=
K(q)\,dq,
\]
where $K(q)$ is the holomorphic continuation of the curvature data.
\section{Dual Conformal Nets}
Let
\[
z=z(q)
\]
be a conformal transformation. For convenience we call the conformal net associated with the inverse mapping
\[
q=q(z)
\]
the dual conformal net. The two nets are symmetric under the interchange
\[
q\leftrightarrow z,
\]
and their conformal factors satisfy
\[
\boxed{\widetilde H=H^{-1}.}
\]
Let
\[
\omega=-(\omega^1_{21}-i\omega^1_{22}),\qquad \widetilde\omega=-(\widetilde\omega^1_{21}-i\widetilde\omega^1_{22})
\]
denote the complex curvature coefficients of the two dual nets.
\begin{proposition}[Dual curvature ratio]
The derivative of the conformal transformation satisfies
\[
\boxed{z_q=-\overline{\left(\frac{\widetilde\omega}{\omega}\right)}.}
\]
\end{proposition}
\begin{proof}.
Using the logarithmic representation  obtained in Lemma 1,
\[
\omega=-\frac{2i}{H}\partial_q\log H,
\qquad
\widetilde{\omega}
=-\frac{2i}{\widetilde H}\partial_z\log\widetilde H,
\]
together with
\[
\widetilde H=H^{-1},
\qquad
\partial_z=z_q^{-1}\partial_q,
\]
we obtain
\[
\frac{\widetilde\omega}{\omega}
=
-\frac{H^2}{z_q}.
\]
Since
\[
H^2=|z_q|^2=z_q\,\overline{z_q},
\]
it follows that
\[
\frac{\widetilde\omega}{\omega}
=
-\overline{z_q},
\]
or equivalently
\[
z_q
=
-\overline{\left(\frac{\widetilde\omega}{\omega}\right)}.
\]
\end{proof}
\begin{remark}
Neither \(\omega\) nor \(\widetilde\omega\) is holomorphic in general. Nevertheless,
\[
\boxed{-\overline{\left(\frac{\widetilde\omega}{\omega}\right)}}
\]
is holomorphic, since it coincides with the derivative of the conformal transition map. Thus holomorphicity emerges from the relation between two geometric curvature fields rather than from either field separately.
\end{remark}
\section{The Elliptic Coordinate Net}

The confocal elliptic coordinate system provides an explicit realization of the general construction.
It is generated by the holomorphic mapping
\[
z=c\cos q,
\qquad
q=q^1+i q^2,
\]
where \(c>0\) is the focal distance.

Writing \(z=x+iy\), we obtain
\[
x=c\cos q^1\cosh q^2,
\qquad
y=-c\sin q^1\sinh q^2.
\]

The derivative is
\[
z_q=-c\sin q,
\]
and therefore the conformal factor is
\[
H=|z_q|
=c|\sin q|.
\]

Using the identity
\[
|\sin(q^1+i q^2)|^2
=
\frac{\cosh2q^2-\cos2q^1}{2},
\]
we obtain
\[
H^2
=
\frac{c^2}{2}
\left(
\cosh2q^2-\cos2q^1
\right).
\]

Hence the metric takes the conformal form
\[
ds^2
=
H^2
\left(
(dq^1)^2+(dq^2)^2
\right).
\]

The orthonormal coframe is
\[
\theta^1=H\,dq^1,
\qquad
\theta^2=H\,dq^2.
\]

Differentiating the identity for \(H^2\), we obtain
\[
2HH_{q^1}
=
c^2\sin2q^1,
\qquad
2HH_{q^2}
=
c^2\sinh2q^2.
\]

Therefore,
\[
\omega^{1}{}_{21}
=
\frac{H_{q^2}}{H^2}
=
\frac{c^2\sinh2q^2}{2H^3},
\]
and
\[
\omega^{1}{}_{22}
=
-\frac{H_{q^1}}{H^2}
=
-\frac{c^2\sin2q^1}{2H^3}.
\]

Introducing the complex curvature coefficient
\[
\omega
=
-\left(
\omega^{1}{}_{21}
-
i\omega^{1}{}_{22}
\right),
\]
we obtain
\[
\omega
=
-\frac{c^2}{2H^3}
\left(
\sinh2q^2
+
i\sin2q^1
\right).
\]

On the other hand,
 a direct computation gives
\[
\omega
=
2i\,\partial_q(H^{-1})=-\frac{2i}{H}\,\partial_q\log H,
\]
where
\[
\partial_q
=
\frac12
\left(
\partial_{q^1}
-
i\partial_{q^2}
\right).
\]

Thus the two real curvature coefficients combine naturally into a single complex curvature coefficient. The conformal connection differential
\[
\Omega=H\omega\,dq
\]
is holomorphic, and the elliptic coordinate net provides an explicit realization of the conformal reconstruction formalism.
\section{Analytic Continuation of Curvature}
The elliptic example suggests a natural interpretation of the conformal connection differential.
Let
\[
\gamma=\{q^2=0\}
\]
be a real analytic curve. On the initial curve, the conformal connection differential restricts to 
\[
\Omega|_{q^2=0}=k(q^1)\,dq^1,
\]
where
\[
k(q^1)=\sum_{n=0}^{\infty}a_n(q^1)^n
\]
is a real analytic function. Replacing the real coordinate \(q^1\) by
\[
q=q^1+iq^2,
\]
we obtain the holomorphic continuation
\[
\boxed{K(q)=\sum_{n=0}^{\infty}a_n q^n.}
\]
The corresponding holomorphic differential is
\[
\boxed{\Omega=K(q)\,dq.}
\]
Thus the Frenet curvature 1-form gives rise to a holomorphic connection differential in a neighborhood of the curve.
For a conformal metric
\[
ds^2=H^2|dq|^2,
\]
the complex curvature coefficient is
\[
\omega=-(\omega^1_{21}-i\omega^1_{22}).
\]
The conformal connection differential is
\[
\boxed{\Omega=H\omega\,dq.}
\]
Hence
\[
\boxed{K(q)=H(q)\omega(q).}
\]
On the initial curve \(q^2=0\) we choose the arc-length normalization
\[
H=1.
\]
Therefore
\[
\Omega|_{q^2=0}=k\,dq^1=k\,ds.
\]
Thus the holomorphic differential
\[
\boxed{\Omega=K(q)\,dq=H\omega\,dq}
\]
is the analytic continuation of the Frenet curvature 1-form.
\begin{remark}
The emphasis is local. Questions of global continuation, monodromy, singularities and natural boundaries are beyond the scope of the present work.
\end{remark}
\section{Reconstruction Formula}
Let
\[
z=z(q)
\]
be a conformal coordinate transformation. Its derivative may be written as
\[
z_q=H e^{i\phi},
\]
where \(H>0\) is the conformal factor and \(\phi\) is the local rotation angle. Therefore,
\[
\log z_q=\log H+i\phi,
\]
and
\[
d\log z_q=d\log H+i\,d\phi.
\]
Since \(z_q\) is holomorphic,
\[
u=\log H,\qquad v=\phi
\]
satisfy the Cauchy--Riemann equations
\[
u_{q^1}=v_{q^2},\qquad u_{q^2}=-v_{q^1}.
\]
Equivalently,
\[
d\phi=*\,d\log H,
\]
where \(*\) denotes rotation by \(+\pi/2\) on one-forms. On the other hand,
\[
\omega^1_2=-*\,d\log H.
\]
Hence
\[
d\log z_q=d\log H+i*d\log H.
\]
Using the definition
\[
\Omega=H\omega\,dq,
\]
we obtain
\[
\boxed{d\log z_q=i\Omega.}
\]
Integrating,
\[
\boxed{z_q=\exp\left(i\int\Omega\right).}
\]
Since \(\Omega\) is the analytic continuation of the Frenet curvature 1-form, this formula reconstructs the conformal coordinate net directly from the continued connection data.
\begin{theorem}[Conformal Reconstruction]
Let \(k(s)\) be the curvature of a real analytic plane curve and let \(K(q)\) be its holomorphic continuation. Then the holomorphic differential
\[
\Omega=K(q)\,dq
\]
determines locally a conformal coordinate net through
\[
\boxed{z_q=\exp\left(i\int\Omega\right).}
\]
\end{theorem}
\begin{remark}[Orientation ambiguity]
The reconstruction is naturally determined only up to orientation. Replacing
\[
\Omega\longrightarrow-\Omega
\]
yields
\[
z_q\longrightarrow\exp\left(-i\int\Omega\right),
\]
which corresponds to reflection of the conformal net. This is analogous to the classical reconstruction of plane curves.
\end{remark}
\begin{remark}[Euler reconstruction]
For the unit circle,
\[
k\equiv1.
\]
Therefore
\[
T(s)=\exp\left(i\int_0^s k\,d\sigma\right)=e^{is}.
\]
Hence
\[
T(0)=1,\qquad T\!\left(\frac{\pi}{2}\right)=i,\qquad T(\pi)=-1.
\]
In particular,
\[
e^{i\pi}=-1.
\]
Thus Euler's formula may be interpreted geometrically as the reconstruction of the tangent vector from the Frenet curvature 1-form. The conformal reconstruction formula
\[
z_q=\exp\left(i\int\Omega\right)
\]
may be viewed as a two-dimensional complex analogue of Euler's formula.
\end{remark}
\begin{figure}[H]
\centering
\begin{tikzpicture}[scale=2.2,>=stealth]
\draw[->] (-1.4,0) -- (1.4,0) node[right] {$\operatorname{Re}$};
\draw[->] (0,-0.2) -- (0,2.3) node[above] {$\operatorname{Im}$};
\draw (0,1) circle (1);
\coordinate (A) at (0,0);
\coordinate (B) at (1,1);
\coordinate (C) at (0,2);
\draw[->,thick] (A) -- (0.65,0) node[below] {$T(0)=1$};
\draw[->,thick] (B) -- (1,1.65) node[right] {$T(\pi/2)=i$};
\draw[->,thick] (C) -- (-0.65,2) node[left] {$T(\pi)=-1$};
\fill (A) circle (0.02);
\fill (B) circle (0.02);
\fill (C) circle (0.02);
\node[below left] at (A) {$s=0$};
\node[right] at (B) {$s=\pi/2$};
\node[above] at (C) {$s=\pi$};
\fill (0,1) circle (0.015);
\node[left] at (0,1) {$i$};
\end{tikzpicture}
\caption{Reconstruction of the tangent vector \(T(s)=e^{is}\) on the unit circle.}
\label{fig:euler-reconstruction}
\end{figure}
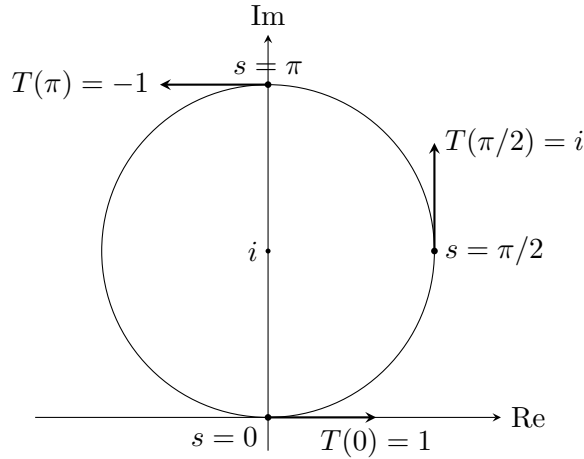
\section{Conclusion}
Classical geometry traditionally begins with coordinate nets adapted to geometry. Cartan geometry, on the other hand, begins with connection forms. In the present work we considered the Frenet curvature 1-form
\[
\omega_F=-k\,ds
\]
of a real analytic plane curve as the primary geometric object. Through analytic continuation this connection
 form gives rise to the holomorphic differential
\[
\Omega=K(q)\,dq=H\omega\,dq.
\]
The resulting differential determines locally a conformal coordinate net through the reconstruction
 formula
\[
z_q=\exp\left(i\int\Omega\right).
\]
Equivalently, the derivative of the conformal transition map admits the geometric representation
\[
z_q=-\overline{\left(\frac{\widetilde\omega}{\omega}\right)}.
\]
Although neither \(\omega\) nor \(\widetilde\omega\) is holomorphic in general, their conjugate normalized 
dual ratio is holomorphic. Thus an analytic plane curve determines, through the holomorphic continuation of its Frenet curvature 1-form, a distinguished conformal coordinate net in a neighborhood of the curve. The present work suggests that, for analytic plane curves, the classical viewpoint based on coordinate nets and Cartan's viewpoint based on connection forms are related through analytic continuation and conformal reconstruction.

\end{document}